\documentclass{amsart}
\usepackage[utf8]{inputenc}
\usepackage[T1]{fontenc}

\usepackage{amsmath}
\usepackage{amsfonts}
\usepackage{todonotes}
\usepackage{amssymb}
\usepackage{amsmath}
\usepackage{amsthm}
\usepackage{framed}
\usepackage[boxed]{algorithm2e}
\usepackage[hidelinks]{hyperref}
\usepackage{multirow,bigdelim}
\usepackage[babel=true]{csquotes}

\newtheorem{lemma}{Lemma}

\newtheorem{theorem}[lemma]{Theorem}

\newcommand{\Z}{\mathbb{Z}}
\newcommand{\Q}{\mathbb{Q}}
\newcommand{\R}{\mathbb{R}}
\newcommand{\F}{\mathbb{F}}
\newcommand{\C}{\mathcal{C}}

\newcommand{\gothp}{\mathfrak{p}}

\newcommand{\dd}{\tilde{d}}
\newcommand{\e}{\tilde{e}}
\newcommand{\softO}{\widetilde{O}}

\DeclareMathOperator{\Res}{Res}
\DeclareMathOperator{\End}{End}
\DeclareMathOperator{\Jac}{Jac}
\DeclareMathOperator{\GCD}{GCD}
\DeclareMathOperator{\maxdeg}{\mathrm{maxdeg}}
\DeclareMathOperator{\Ker}{Ker}

\begin{document}

\title[Counting points on genus-3 hyperelliptic
curves with explicit RM]
{Counting points on genus-3 hyperelliptic
curves with explicit real multiplication}

\iftrue
\author{Simon Abelard}
\address{Université de Lorraine, CNRS, Inria}
\email{simon.abelard@loria.fr}

\author{Pierrick Gaudry}
\email{pierrick.gaudry@loria.fr}

\author{Pierre-Jean Spaenlehauer}
\email{pierre-jean.spaenlehauer@loria.fr}
\else
\author{}
\fi

\date{}

\begin{abstract} We propose a Las Vegas probabilistic algorithm to compute the
  zeta function of a genus-3 hyperelliptic curve defined over a finite field
  $\mathbb F_q$, with
  explicit real multiplication by an order $\mathbb Z[\eta]$ in a totally real cubic
  field.
  Our main result states that this algorithm requires
  an expected number of $\softO((\log q)^6)$ bit-operations, where the
  constant in the $\softO()$ depends on the ring $\mathbb Z[\eta]$ and on the degrees of polynomials
  representing the endomorphism $\eta$.
  As a proof-of-concept, we compute the zeta function of a curve
  defined over a 64-bit prime field, with explicit real multiplication by
  $\Z[2\cos(2\pi/7)]$.
\end{abstract}

\maketitle

\thispagestyle{empty}
\section{Introduction}

Since the discovery of Schoof's algorithm~\cite{Sc85}, the problem of computing
efficiently zeta functions of curves defined
over finite fields has attracted a lot of attention, as its applications range
from the construction of cryptographic curves to testing conjectures
in number theory. 
We focus on the problem of computing the zeta function of a hyperelliptic curve
$\C$ of
genus~3 defined over a finite field $\mathbb F_q$ using $\ell$-adic methods, in the spirit of Schoof's algorithm and
its generalizations \cite{Pi90,HI98,AH01}. Although
these methods are polynomial with respect to $\log q$, the exponents in the best known complexity bounds grow quickly
with the genus. Another line of research is to use $p$-adic
methods~\cite{Ked01,Satoh00,carls2009p,harrison2012extension}, which are polynomial in the
genus but exponential in the size of the characteristic of the
underlying finite field. Variants of these
methods~\cite{KedSuth08,Harvey14,HarSuth16} allow to count the
points of a curve defined over the rationals modulo many primes in
average polynomial time, which is especially relevant when experimenting
with the Sato-Tate conjecture.

The aim of this paper is to show --- both with theoretical proofs and practical
experiments --- that the complexity of $\ell$-adic methods for
genus-$3$ hyperelliptic curves can be dramatically decreased as soon as an
explicitly computable non-integer
endomorphism $\eta\in \End(\Jac(\C))$ is known. More precisely, we say that a
curve $\C$ has \emph{explicit real multiplication} by $\mathbb Z[\eta]$ if the subring 
$\mathbb Z[\eta]\subset \End(\Jac(\C))$ is isomorphic to an order in a totally
real cubic number field, and if we have explicit formulas describing
$\eta(P-\infty)$ for some fixed base point $\infty$ and a
generic point $P$ of $C$. By explicit formulas, we mean polynomials
$(\eta^{(u)}_i(x,y))_{i\in\{0,1,2,3\}}$ and
$(\eta^{(v)}_i(x,y))_{i\in\{0,1,2,3\}}$ in $\mathbb F_q[x,y]$,
 such that, when $\C$
is given in odd-degree Weierstrass form, the Mumford coordinates of $\eta( (x, y)-\infty)$ are
$\Big\langle \sum_{i=0}^3
{\eta^{(u)}_i(x,y)}X^i, \sum_{i=0}^2 
\big({\eta^{(v)}_i(x,y)}/{\eta^{(v)}_3(x,y)}\big)X^i\Big\rangle,$ where $(x,y)$ is the generic point
of the curve. In cases where $\C$ does not have an odd-degree Weierstrass
model, we can work in an extension of degree at most $8$ of the base
field in order to ensure the existence of a rational Weierstrass point.

The influence of real multiplication on the complexity of
point counting was investigated for genus 2 curves in~\cite{GKS11}, where
the authors decrease the complexity from $\softO((\log q)^8)$~\cite{GS12} to
$\softO( (\log q)^5 )$.
For genus $2$ curves, another related active line of research is to mimic the
improvement of Elkies and Atkin by using modular
polynomials~\cite{ballentine2017isogenies}. However, the main difficulty of
this method is to precompute the modular polynomials, which are much larger
than their genus $1$ counterparts.

Our main result is the following theorem.
\begin{theorem} 
  Let $\C$ be a genus-3 hyperelliptic curve defined over a finite field
  $\mathbb F_q$ having explicit real multiplication by $\mathbb Z[\eta]$, where
  $\eta\in\End(\Jac(C))$. We assume
  that $\C$ is given by an odd-degree Weierstrass equation $Y^2=f(X)$.  The
  characteristic polynomial of the Frobenius endomorphism on the Jacobian of
  $\C$ can be computed with a Las Vegas probabilistic algorithm in expected
  time bounded by $c\,(\log q)^6(\log\log q)^k$, where $k$ is an absolute
  constant and $c$ depends only on the
  degrees of the polynomials $\eta_i^{(u)}$ and $\eta_i^{(v)}$ and on the ring
  $\mathbb Z[\eta]$.
\end{theorem}

In this paper, we use the notation $\softO()$ as a shorthand for complexity
statements hiding poly-logarithmic terms: the complexity in the theorem would
be abbreviated $\softO((\log q)^6)$. We insist on the fact that all
the $O()$ and the $\softO()$ notation used throughout the paper should be
understood up to a multiplicative constant which may depend on the ring
$\Z[\eta]$ and on the degrees of the polynomials $\eta_i^{(u)}$ and
$\eta_i^{(v)}$. There are natural families of
curves for which these degrees are bounded by an absolute constant and for
which $\mathbb Z[\eta]$ is fixed: reductions at primes (of good reduction) of a
hyperelliptic curve with explicit RM defined over a number field.

As in Schoof's algorithm and its generalizations in~\cite{Pi90,HI98,AH01}, the $\ell$-adic approach consists in computing the characteristic polynomial of the
Frobenius endomorphism by computing its action on the $\ell$-torsion of
the Jacobian of the curve for sufficiently many $\ell$.
In order to prove the claimed complexity bound, we consider primes
$\ell\in\mathbb Z$ such that $\ell\mathbb Z[\eta]$
splits as a product $\gothp_1\gothp_2\gothp_3$ of prime ideals. Computing the kernels of endomorphisms $\alpha_i$
in each $\gothp_i$ provides us with an algebraic representation of
the $\ell$-torsion $\Jac(\C)[\ell]\subset \Ker
\alpha_1 + \Ker\alpha_2 + \Ker\alpha_3$. Then, we compute from this representation
integers $a,b,c\in\Z/\ell\Z$ such that
the sum $\pi+\pi^{\vee}$ of the Frobenius endomorphism and its dual
equals $a+b\eta+c\eta^2\bmod
\ell$. Once enough modular information is known, the values of $a$, $b$,
$c$ such that $\pi+\pi^{\vee} = a+b\eta+c\eta^2$ are recovered 
via the Chinese Remainder Theorem and 
the coefficients of the characteristic polynomial of the Frobenius
can be directly expressed in terms of $a$, $b$ and $c$.
In fact, in practice we do not have to restrict
to split primes:
any partial
factorization of $\ell\mathbb Z[\eta]$ provides some modular information
on $a, b, c\bmod \ell$.
We
give an example with a ramified prime in Section~\ref{sec:modinfoGB}; but on
the theoretical side, considering non-split primes does not improve the
asymptotic complexity.

The cornerstone of the complexity analysis is the cost of the computation of the kernels of the
endomorphisms. This is achieved by solving a polynomial system. Using
resultant-based elimination techniques and degree bounds on Cantor's
polynomials, we prove that we can solve these
equations in time quadratic in the number of solutions, which leads to the
claimed complexity bound. For practical
computations, we replace the resultants by Gröbner bases
and we retrieve modular information only for small $\ell$ to speed up an
exponential collision search which can be massively run in parallel.
Although using Gröbner basis seems to be more efficient in practice,
we do not see any
hope of proving with rigorous arguments that it is
asymptotically competitive.

As a proof-of-concept, we have implemented our algorithm and we provide
experimental results. In particular, we were able to compute the zeta function
of a genus $3$ hyperelliptic curve with explicit RM defined over $\F_p$ with
$p=2^{64}-59$. To our knowledge the largest
genus-$3$ computation that had been achieved previously was the computation of the zeta
function of a hyperelliptic curve defined over $\F_p$ with 
$p=2^{61}-1$, done by Sutherland~\cite{SuthComp}
using generic group methods.

Examples of curves with RM are given by modular curves. For instance, the
genus-3 curve $y^2 = x^7+3x^6+2x^5-x^4-2x^3-2x^2-x-1$ is a quotient of
$X_0(284)$ and therefore has real multiplication by an element of
$\Q[x]/(x^3-3x-1)$.  This follows from the properties of the Hecke
operators as explained in~\cite[Chapter 7]{Shimura71}.
Based on this theory, algorithms for constructing such curves are explained
in~\cite{FreyMuller}; however the explicit expression for the real
endomorphism is not given.  We expect that tracking the Hecke
correspondences along their construction, and using techniques like
in~\cite{Wamelen99} to reconstruct the rational fractions describing the
real endomorphism could solve this question.
In any case, these are only isolated points in the moduli space. Larger
families are obtained from cyclotomic covering. This line of research has
produced several families of hyperelliptic genus-3 curves having explicit RM
by $\mathbb Z[2\cos(2\pi/7)]$. In particular, explicit such families
are given in \cite{Mestre} and \cite{TTV}, and explicit formulas for their RM
endomorphism are obtained in~\cite{KoSm}. We use the 1-dimensional family of
curves from \cite[Theorem~1 with $p=7$]{TTV} for our experiments.
Other families of genus-3 curves (but not necessarily hyperelliptic) with
RM have been made explicit in~\cite[Chapter~2]{boyer2014varietes},
following~\cite{ellenberg2001}.
We would like
to point out that within the moduli space of complex polarized abelian
varieties of dimension $3$, those with
RM by a fixed order in a cubic field form a moduli space of codimension
$3$~\cite[Sec.~9.2]{BirkLange}. Since Jacobians of hyperelliptic
curves form a codimension $1$ space, we would expect the moduli space of
hyperelliptic curves of genus $3$ with RM by a given cubic order to have
dimension $2$.

We finally briefly mention how our algorithm and analysis could be
extended in several directions. First, the complexity analysis leads,
with small modifications, to a point-counting algorithm for general
genus-3 hyperelliptic curves (i.e.  without RM) with complexity in
$\softO( (\log q)^{14} )$. Second, if the curve is not hyperelliptic, the
main difficulty is to define analogues of Cantor's division
polynomials and get bounds on their degrees. Without them, it is still
possible to use an explicit group law to derive a polynomial system for
the kernel of an endomorphism, but getting a proof for its degree would
require to take another path than what we did. Still, the complexities
with or without RM
are expected to remain the same for plane quartics as for genus-3
hyperelliptic curves. Third, if we go to higher genus hyperelliptic
curves with RM, the main difficulty to extend our approach is in the
complexity estimate of the polynomial system solving, because
resultant-based approaches are not competitive when the number of
variables grows, and a tedious analysis like in~\cite{expOg} seems to
be necessary.

The article is organized as follows.
Section~\ref{sec:overview} gives a bird-eye view of our algorithm, along with a complexity
analysis relying on the technical results detailed in
Sections~\ref{sec:kernel} to \ref{sec:cantor}.
Practical experiments are presented in Section~\ref{sec:experiments}. 

\subsection*{Acknowledgements.} We are grateful to Benjamin Smith for
fruitful discussions and to Allan Steel for his help with memory issues
with Magma. We also wish to thank anonymous referees for their comments
which helped improve the paper.

\section{Overview of the algorithm}\label{sec:overview}

Let $\C$ be a genus-3 hyperelliptic curve over a finite field $\F_q$ with
explicit RM, and let $\eta$ be the given explicit endomorphism. We denote by
$\mu_0, \mu_1, \mu_2$ the coefficients of the minimal polynomial $T^3+
\mu_2\, T^2 + \mu_1\, T + \mu_0$ of $\eta$ over $\mathbb Q$.

\subsection{Bounds}

The characteristic polynomial of the
Frobenius endomorphism $\pi$ is of the form
$\chi_{\pi}(T)=T^6-\sigma_1T^5+\sigma_2T^4-\sigma_3T^3+q\sigma_2T^2-q^2\sigma_1T+q^3$,
and Weil's bounds give
$$ |\sigma_1| \le 6\sqrt{q},\quad |\sigma_2|\le 15 q,\quad |\sigma_3|\le
20 q^{3/2}.$$

In order to take advantage of the explicit RM, we consider
the endomorphism $\psi=\pi+\pi^{\vee}$, for which we can derive the
real Weil's polynomial 
$\chi_{\psi}(T) =
T^3-\sigma_1T^2+(\sigma_2-3q)T-(\sigma_3-2q\sigma_1)$, which corresponds to the
characteristic polynomial of $\psi$ viewed as an element of the real subfield
of $\End(\Jac(C))\otimes\mathbb Q$.
The endomorphism $\psi$ belongs to the ring of integers of $\Q(\eta)$.
The ring $\Z[\eta]$ might be a proper sub-order of the ring of integers,
so let us call $\Delta$ its index, so that $\psi$ can be written
$\psi=a+b\eta+c\eta^2$, where $a$, $b$, $c$ are rationals with a
denominator that divides $\Delta$. By computing formally the characteristic
polynomial of $a+b\eta+c\eta^2$ in $\Q(\eta)$ and by equating it with the
expression for the real Weil's polynomial $\chi_\psi(T)$, we obtain a direct way to
compute $\sigma_1$, $\sigma_2$ and $\sigma_3$ in terms of $a$, $b$, $c$:
\begin{small}
\begin{equation}\label{eq:relabcs}
    \begin{array}{rcl}
        \sigma_1 & = & 3\,a - b\,\mu_2 - 2\,c\,\mu_1 + c\,\mu_2^2\ , \\
        \sigma_2 - 3q & = & 3\,a^2 - 2\,a\,b\,\mu_2
        + 2\,a\,c\,(\mu_2^2 - 2\mu_1)
        + b^2\,\mu_1 + 3\,b\,c\,\mu_0 - 
	 b\,c\,\mu_1\,\mu_2\, - \\
         && c^2\,(2\,\mu_0\,\mu_2 + \mu_1^2)\ , \\
        \sigma_3 -2q\sigma_1 & = & a^3 - a^2\,b\,\mu_2
        + a^2\,c\,(\mu_2^2 -2\mu_1)
        + a\,b^2\,\mu_1 + 
        a\,b\,c\,(3\,\mu_0 - \mu_1\,\mu_2)\, + \\ &&
         a\,c^2\,(\mu_1^2 - 2\, \mu_0\,\mu_2)
        - b^3\,\mu_0 + b^2\,c\,\mu_0\,\mu_2 -
	b\,c^2\,\mu_0\,\mu_1 + c^3\,\mu_0^2\ .
    \end{array}
\end{equation}
\end{small}

In Section~\ref{sec:boundabc}, it is shown that the coefficients $a$, $b$
and $c$ can be bounded in $O(\sqrt{q})$. More precisely, we denote by
$C_{abc}$ a constant that depends only on $\eta$ such that their
absolute values are bounded by
$C_{abc}\sqrt{q}$.
Since these bounds are much smaller than the bounds for
$\sigma_1$, $\sigma_2$, $\sigma_3$, it makes sense to design an algorithm
that reconstruct these coefficients of $\psi$ instead of the coefficients
of $\chi_\pi$ as in the classical Schoof algorithm, and this is what we
are going to do later on.
\medskip

Another important bound that we need concerns the size of small elements
that can be found in ideals of $\Z[\eta]$.
Let $\ell$ be a prime that splits completely
in $\Z[\eta]$, so that we can write
$ \ell = \gothp_1 \gothp_2 \gothp_3$,
where the $\gothp_i$'s are distinct prime ideals of norm $\ell$. In
Section~\ref{sec:smallgen}, it is shown that each $\gothp_i$ contains a
non-zero element $\alpha_i = a_i + b_i\eta +c_i\eta^2$,
where $a_i$, $b_i$ and $c_i$ are integers and are bounded in absolute
value by $O(\ell^{1/3})$.

\subsection{Algorithms}

The general RM point counting algorithm is Algorithm~\ref{algo:rm}.  We
give a description of it, allowing some black-box primitives that
will be detailed in dedicated sections.
As mentioned above, we will work with the $a$, $b$, $c$ coefficients of the
$\psi$ endomorphism. More precisely, we compute their values modulo
sufficiently many completely split primes $\ell$ until we can deduce their
values from the bounds of Lemma~\ref{lem:boundabc} by the Chinese
Remainder Theorem, taking into account their potential denominator
$\Delta$. Then the coefficients of $\chi_\pi$ are deduced by
Equations~\eqref{eq:relabcs}.

We now explain how the algorithm works for a given split $\ell$. First
its decomposition as a product of prime ideals $\ell\,
\Z[\eta]=\gothp_1\gothp_2\gothp_3$ is computed, and for each prime ideal
$\gothp_i$, a non-zero element $\alpha_i$ of $\gothp_i$ is found with a
small representation
$\alpha_i=a_i+b_i\eta +c_i\eta^2$ as in Lemma~\ref{lem:smallgen}. In fact,
$\gothp_i$ is not necessarily principal and $\alpha_i$ need not generate
$\gothp_i$. The
kernel of $\alpha_i$ is denoted by $J[\alpha_i]$ and it contains a subgroup 
$G_i$ isomorphic to $\Z/\ell\Z\times\Z/\ell\Z$, since the norm of $\alpha_i$ is
a multiple of $\ell$. The two-element representation $(\ell,
\eta-\lambda_i)$ of the ideal $\gothp_i$ implies that $\lambda_i$ is an
eigenvalue of $\eta$ regarded as an endomorphism of $J[\ell]\cong(\mathbb
Z/\ell\mathbb Z)^6$.

On $G_i\subset J[\alpha_i]$, the endomorphism $\eta$ acts as the
multiplication by
$\lambda_i$. Therefore, $\psi = a+b\eta+c\eta^2$ also acts as a scalar
multiplication on this 2-dimensional space, and we write
$k_i\in\Z/\ell\Z$ the corresponding eigenvalue: for any $D_i$ in
$G_i$, we have $\psi(D_i) = k_iD_i$. On the other hand,
from the definition of
$\psi$, it follows that $\psi\pi = \pi^2+q$.
Therefore,
if such a $D_i$ is known, we can test which value of $k_i \in
\Z/\ell\Z$ satisfies
\begin{equation}\label{eq:pi}
 k_i\pi(D_i) = \pi^2(D_i) + qD_i.
\end{equation}
Since $\ell$ is a
prime and $D_i$ is of order exactly $\ell$, this is also the case for
$\pi(D_i)$. Finding $k_i$ can then be seen as a discrete logarithm
problem in the subgroup of order $\ell$ generated by $\pi(D_i)$; hence the solution is
unique. Equating the two expressions for $\psi$, we get 
explicit relations between $a$, $b$, $c$ modulo $\ell$:
 $$ a + b\lambda_i + c\lambda_i^2 \equiv k_i \bmod \ell.$$
Therefore we have a linear system of three equations in three
unknowns, the determinant of which is the Vandermonde determinant of
the $\lambda_i$, which are distinct by hypothesis. Hence the system can
be solved and it has a unique solution modulo $\ell$.

\SetAlCapSkip{1ex}

\begin{algorithm}[ht] \label{algo:rm}
  \KwData{$q$ an odd prime power, and $f\in\mathbb F_q[X]$ a monic squarefree
    polynomial of degree 7 such that the curve $Y^2=f(X)$ has explicit
    RM by $\Z[\eta]$.}
  \KwResult{The characteristic polynomial $\chi_\pi\in\mathbb Z[T]$ of the
    Frobenius endomorphism on the Jacobian $J$ of the curve.}
 
  $R\gets 1$\;

  \While{$R\leq 2\,\Delta\,C_{abc}\,\sqrt{q} + 1$} {
    Pick the next prime $\ell$ that satisfies conditions (C1) to (C4);

    Compute the ideal decomposition
    $\ell\, \Z[\eta]=\gothp_1\gothp_2\gothp_3$, corresponding to the
    eigenvalues $\lambda_1$, $\lambda_2$, $\lambda_3$ of $\eta$ in
    $J[\ell]$ \;

    \For{$i\gets1$ \KwTo $3$}{
      Compute a small element $\alpha_i$ of $\gothp_i$ as in
      Lemma~\ref{lem:smallgen}\;

      Compute a non-zero element $D_i$ of order $\ell$ in $J[\alpha_i]$ \;

      Find the unique $k_i\in\Z/\ell\Z$ such that
      $k_i\pi(D_i) = \pi^2(D_i) + qD_i$ \;
    }

    Find the unique triple $(a,b,c)$ in $(\Z/\ell\Z)^3$ such that
    $ a + b\lambda_i + c\lambda_i^2 = k_i$, for $i$ in $\{1,2,3\}$ \;

    $R \gets R\cdot\ell$\;
  }
    Reconstruct $(a,b,c)$ using the Chinese Remainder Theorem \;

    Deduce $\chi_{\pi}$ from Equations~\eqref{eq:relabcs}.

    \caption{Overview of our RM point-counting algorithm}
\end{algorithm}

It remains to show how to construct a divisor $D_i$ in $G_i$, i.e. an
element of order $\ell$ in the kernel
$J[\alpha_i]$.  Since an explicit expression of $\eta$ as an endomorphism
of the Jacobian of $\C$ is known, an explicit expression can be deduced for $\alpha_i$,
using the explicit group law. The coordinates of the elements of this
kernel are solutions of a polynomial system that can be directly derived
from this expression of $\alpha_i$.  Using standard techniques, it is
possible to find the solutions of this system in a finite extension
of the base field (of degree bounded by the degree of the ideal generated by the
system, i.e. in $O(\ell^2)$), from which divisors in $J[\alpha_i]$ can be
constructed.
Multiplying by the appropriate cofactor, we can reach all the elements of
$G_i$; but we stop as soon as we get a non-trivial one.

We summarize the conditions that must be satisfied by the primes $\ell$ that we
work with:
\begin{enumerate}
  \item[(C1)] $\ell$ must be different from the characteristic of the base field;
  \item[(C2)] $\ell$ must be coprime to the discriminant of the minimal
    polynomial of $\eta$;
  \item[(C3)] there must exist $\alpha_i\in\mathfrak p_i$ as in
    Lemma~\ref{lem:smallgen} with norm non-divisible by $\ell^3$ for
    $i\in\{1,2,3\}$;
  \item[(C4)] the ideal $\ell\Z[\eta]$ must split completely.
\end{enumerate}
The first 3 conditions eliminate only a finite number of $\ell$'s that
depends only on $\eta$,
while the last one eliminates a constant proportion. The
condition (C3) implies that there is a unique subgroup $G_i$ of order
$\ell^2$ in $J[\alpha_i]$ (our description of the algorithm could actually be
adapted to handle the cases where this is not true).

Algorithm~\ref{algo:rm} is a very natural extension of the one described
in~\cite{GKS11} for genus 2 curves with RM. Already in~\cite{GKS11}, the
action of the real endomorphism $\psi=\pi+\pi^\vee$ is studied on
subspaces $J[\gothp_i]$ of the $\ell$-torsion, and the corresponding
eigenvalues are collected and used to reconstruct information modulo
$\ell$. In genus 3, we have 3 such $2$-dimensional subspaces and
eigenvalues to compute and recombine instead of~2 in genus 2. The main
differences between the present work and~\cite{GKS11} are the way the
$\ell$-torsion elements are constructed with polynomial systems and
the bounds on the coefficients of $\psi$. In both cases, going from
dimension 2 to 3 is not immediate.

\subsection{Complexity analysis}

The field $\Q(\eta)$ is of degree 3, so its Galois group has order at most $6$
and by Chebotarev's density theorem the density of primes
that split completely is at least $1/6$. Therefore the main loop is done
$O(\log q/\log\log q)$ times, with primes $\ell$ that are in $O(\log q)$.
All the steps that take place in the number field take a negligible time.
For instance, a small generator like in Lemma~\ref{lem:smallgen}
can be found by exhaustive search: only $O(\ell)$ trials are needed since we
are searching over all elements of the form $a + b\eta +c\eta^2$, with $\lvert
a\rvert,\lvert b\rvert,\lvert c\rvert$ in
$O(\ell^{1/3})$.

The bottleneck of the algorithm is the computation of a non-zero element 
of order $\ell$ in the kernel $J[\alpha_i]$ of $\alpha_i$.
This part will be treated in detail in Section~\ref{sec:kernel}, where
it is shown to be feasible in $\softO(\ell^4)$ operations in
$\F_q$. The output is a divisor $D_i$ of order $\ell$ in $J[\alpha_i]$ that is defined
over an extension field $\F_{q^\delta}$, where $\delta$ is in
$O(\ell^2)$.

In order to check Equation~\eqref{eq:pi}, we first need to compute
$\pi(D_i)$ and $\pi^2(D_i)$ which amounts to raising the coordinates to
the $q$-th power. The cost is in $\softO(\ell^2\log q)$ operations in
$\F_q$. Then, each Jacobian operation in the group generated by
$\pi(D_i)$ costs $\softO(\ell^2)$ operations in the base field, and we
need $O(\sqrt{\ell})$ of them to solve the discrete logarithm problem
given by Equation~\eqref{eq:pi}. The overall cost of finding $k_i$, once
$D_i$ is known is therefore $\softO(\ell^2(\sqrt{\ell}+\log q))$
operations in $\F_q$.

Finally, the amount of work performed for each $\ell$ is
$\softO(\ell^2(\ell^2+\log q))$ operations in the base field $\F_q$. Summing
up for all the primes, and taking into account the cost of the operations
in $\F_q$, we obtain a global bit-complexity of $\softO((\log q)^6)$.

\section{Computing kernels of endomorphisms}\label{sec:kernel}

\subsection{Modelling the kernel computation by a polynomial system}

Let $\alpha$ be an explicit endomorphism of degree
$O(\ell^2)$ on the Jacobian of $\C$, which satisfies the properties of
Lemma~\ref{lem:smallgen}. In particular, $\alpha$ vanishes on a subspace of
$J[\ell]$.
We want to compute a triangular polynomial system that describes the kernel
$J[\alpha]$ of $\alpha$. This will provide us with a nice description of a
subgroup of the $\ell$-torsion on which we will be able to test the action of
$\psi=\pi+\pi^{\vee}$ and deduce $a$, $b$, $c$ such that
$\psi = a+b\eta+c\eta^2 \bmod \ell$.

We first model $J[\alpha]$ by a system of polynomial equations that we
will then put in triangular form. To do so, we consider a
generic divisor $D=P_1+P_2+P_3-3\infty$, where $P_i$ is an affine point
of $\C$ of coordinates $(x_i,y_i)$. We then write $\alpha(D)=0$,
i.e $\alpha(P_1-\infty)+\alpha(P_2-\infty)=-\alpha(P_3-\infty)$.
Generically, we expect each $\alpha(P_i-\infty)$ to be of weight 3, and we write
$\langle u_i,v_i\rangle$ for its Mumford form. We derive our equations by computing
the Mumford form $\langle u_{12},v_{12}\rangle$ of
$\alpha(P_1-\infty)+\alpha(P_2-\infty)$ and
then writing coefficient-wise the conditions $u_{12}=u_3$ and $v_{12}=-v_3$.
The case where the genericity conditions are not satisfied is discussed
at the end of the section.

Similarly to the Schoof-Pila algorithm, we define polynomials --- which are equivalent to
Cantor's division polynomials --- by the formulas
$$
\begin{array}{c}\displaystyle u_{12}(X)=X^3+\sum_{i=0}^2 \frac{\tilde{d}_i(x_1,x_2,y_1,y_2)}{\tilde{d}_3(x_1,x_2)}X^i,
\quad
v_{12}(X)=\sum_{i=0}^2\frac{\e_i(x_1,x_2,y_1,y_2)}{\e_3(x_1,x_2)}X^i,\\
\displaystyle u_3(X)=X^3+\sum_{i=0}^2 \frac{d_i(x_3)}{d_3(x_3)}X^i,
\quad
v_3(X)=y_3\sum_{i=0}^2\frac{e_i(x_3)}{e_3(x_3)}X^i.
\end{array}
$$

\begin{lemma}
  For any $i\in \{ 1,2,3\}$, the degrees of $\dd_i$, $\e_i$, $d_i$ and
  $e_i$ are in $O(\ell^{2/3})$.
\end{lemma}

\begin{proof}
  
  Let us first remark that the $\dd_i$'s and $\e_i$'s are obtained after
  adding two divisors $\langle u_{1},v_{1}\rangle$ and $\langle
  u_{2},v_{2}\rangle$ such that the coefficients of the $u_i$ and $v_i$
  are respectively the $d_j/d_3$ and $y_ie_j/e_3$ evaluated at $x_i$.
  Thus, since this application of the group law involves a number of
  operations that is bounded independently of $\ell$ and $q$, the degree stays
  within a constant multiplicative factor, which is captured by the
  $O()$. Therefore it is
  enough to prove the result for the $d_i$'s and $e_i$'s.

  Since the endomorphism $\alpha$
  satisfies the properties of Lemma~\ref{lem:smallgen}, it is a linear
  combination of $1$, $\eta$ and $\eta^2$ with coefficients
  of size $O(\ell^{1/3})$.
  Using the same argument about the group law, we can further
  reduce our proof to the case where $\alpha=n\eta^k$, with $k \in \{
    0,1,2\}$ and $n$ an integer in $O(\ell^{1/3})$.
  But once again, $\eta^k$ does not depend
  on $\ell$ so that, provided we can prove that Cantor's $n$-division
  polynomials have degrees in $O(n^2)$, we have proven that
  $n\eta^k(P-\infty)=\eta^k(n(P-\infty))$ have coefficients whose degrees
  are in $O(n^2)$, and then so does $\alpha(P-\infty)$.
  This quadratic bound on the degrees of Cantor's division polynomials is
  proven in Lemma~\ref{lemme10} of Section~\ref{sec:cantor} and the result follows.
\end{proof}

\subsection{Solving the system with resultants}\label{sec:resultants}

Typical tools for solving a polynomial system are the F4 algorithm,
methods based on geometric resolution, or homotopy techniques. To obtain
reasonable complexity bounds, they all require some knowledge of the
properties of the system, and this might be hard to prove. Since we have a
system in essentially 3 variables (in fact, there are six variables $x_1, x_2,
x_3, y_1, y_2, y_3$, but the $y_i$ variables can be directly eliminated by
using the equation defining the curve), we prefer to stick to an approach
based on resultants. It ends up having a complexity that is
quasi-quadratic in the degree of the ideal, which is the best that can be
hoped for anyway for all of the advanced techniques, and the complexity
analysis requires only elementary tools. A complication that can
occur with resultants is that $\Res_x(f,g)$ is identically zero when $f$
and $g$ have a nonconstant GCD. This is not a problem in our case since we
can divide polynomials $f$ and $g$ by their GCD, by factoring them
at the cost of $O(\max(\deg(f), \deg(g))^\omega)$ field operations --- where $\omega \leq 3$
is the exponent of linear algebra --- using the bivariate recombination
methods
in~\cite{factobiv} (the trivariate case can be reduced to the bivariate case by
using the techniques in~\cite[Sec.~21.2]{zippel2012effective}). In what follows, the complexities of
computing the resultants are larger than $O(\max(\deg(f), \deg(g))^\omega)$, so
we can forget about this complication. We also note that since the system is symmetric with respect to
$x_1$ and $x_2$, it may be possible to decrease the degrees by rewriting the system in terms of elementary symmetric
polynomials in $x_1$ and $x_2$; however, we do not consider this symmetrization
process in the analysis since it may only win a constant factor in
the complexity. 

Following our modelling, the equality of the $u$-coordinates gives
three equations
\begin{equation}\label{eq:kernel_endo}
\forall i\in\{0,1,2\}, \quad  \dd_i(x_1,x_2,y_1,y_2)d_3(x_3)=\dd_3(x_1,x_2)d_i(x_3),
\end{equation}
of degree $O(\ell^{2/3})$ in the $x_i$'s. By computing resultants
with the equations $y_i^2=f(x_i)$, we derive three equations
$E_i(x_1,x_2,x_3)=0$ whose degrees are still in $O(\ell^{2/3})$. 

We then eliminate $x_1$ by computing 3 trivariate resultants $R_i$
(between the two equations $E_j$ with $j\ne i$). We get three equations
$R_i(x_2,x_3)=0$ of degrees $O(\ell^{4/3})$ within a complexity in
$\softO(\ell^{10/3})$ field operations, as proven in
Lemma~\ref{lem:trivres} below.

Then, we compute bivariate resultants $S_i$ (between the two equations
$R_j$ with $j\ne i$) to eliminate $x_2$. From Lemma~\ref{lem:bivres}, we
get three univariate equations $S_i(x_3)=0$ of degree bounded by
$O(\ell^{8/3})$ for a complexity in $\softO(\ell^4)$ field operations.
And we compute the polynomial $S(x_3)$ as the GCD of the $S_i(x_3)$,
which belongs to the ideal defined by our original system.

The bound on the degree of $S$ is much larger than $\ell^2-1$, the
expected degree of the kernel. Although we can expect the actual degree
to be in $O(\ell^2)$, we need to add the constraints coming from the
$v$-coordinates to be able to prove it.

The polynomial system coming from $v_{12} = -v_3$ has the same
characteristics as the one coming from the $u$-coordinates. Therefore, we
can proceed in a similar way and deduce, at a cost of $\softO(\ell^4)$
operations another univariate polynomial $\tilde{S}(x_3)$ belonging to
the ideal. Now, since all the original equations have been taken into
account all common roots of $S$ and $\tilde{S}$ will correspond to a
solution of the original system for which we know that there are $O(\ell^2)$
solutions. Therefore taking the squarefree part of the GCD of $S$ and $\tilde{S}$ yields a
polynomial of degree $O(\ell^2)$.

This univariate polynomial can be factored
at a cost of $\softO(\ell^4)$ operations in $\F_q$ with standard
algorithms~\cite{GatGer} (there exist asymptotically faster
algorithms, but we already fit in our target complexity).
We then deal with each irreducible factor in turn, until one is found
that leads to a genuine solution of the original system. Let $\delta$ be
the degree of such an irreducible factor $\phi(x_3)$. In the
field extension $\F_{q^\delta} = \F_q[x_3]/\phi(x_3)$, we have by
construction a root $x_3$ of $\phi$. We then solve again the original
polynomial system where $x_3$ is instantiated with this root. This system
is bivariate in $x_1$ and $x_2$ and there are $O(1)$ solutions, that
possibly live in another finite extension $\F_{q^{\delta'}}$ of
$\F_{q^\delta}$. Since the degrees of the bivariate polynomials are in
$O(\ell^{2/3})$, by Lemma~\ref{lem:bivres}, this system solving costs
$\softO(\ell^2)$ operations in $\F_{q^\delta}$.

A solution obtained in this way must be checked, because it could come
from a vanishing denominator that has been cleared when constructing the
system or from non-generic situations. But given a set of candidate
coordinates for a $D_i$ element of $J[\alpha_i]$, it is cheap to check
that this is indeed an element of the Jacobian and that it is killed by
$\alpha_i$. Also, if $\alpha_i$ is not a generator of $\gothp_i$, it is
necessary to check the order of $D_i$: if this is a multiple of $\ell$,
then multiplying $D_i$ by the cofactor gives an order-$\ell$ element. But
it is also possible to get an unlucky element of small order
coprime to $\ell$, and then we have to take another solution of the
system.

Since an operation in $\F_{q^\delta}$ requires a number of operations in
$\F_q$ that is quasi-linear in $\delta$, and since the sum of all the
degrees $\delta$ of the irreducible factors of $\GCD(S, \tilde{S})$ is in
$O(\ell^2)$, the amortized cost is $\softO(\ell^4)$ operations in $\F_q$
to deduce a divisor $D_i$ in $J[\alpha_i]$.

\subsection{Complexity of bi- and tri-variate resultants}

In this section, the algorithms work by evaluation / interpolation,
which requires to have enough elements in the base field. Were it not the
case, we simply take a field extension $\F_{q^\delta}$ of $\F_q$, that
will add a factor $\softO(\delta)$ to the complexity. The complexity of
the algorithms will be polynomial in the number of evaluation points,
therefore, the final complexity will be logarithmic in $\delta$, so that
the cost of taking a field extension will be hidden in the $\softO()$
notation. We will therefore not mention this potential complication
further.

Another difficulty is that an evaluation / interpolation strategy assumes
that the points of evaluation are generic enough, so that all the degrees
after evaluation are generic. This is again guaranteed by taking a large
enough base field. Still, the algorithm remains a Monte-Carlo one.
However, the ultimate goal is to construct kernel elements, which is an
easily verified property. Turning this into a Las Vegas algorithm can
therefore be done with standard techniques.

\begin{lemma}\cite[Thm.~6.22 and Cor.~11.21]{GatGer}\label{lem:bivres}
  Let $P(x,y)$ and $Q(x,y)$ be two polynomials whose degrees in $x$ and $y$ are
  bounded by $d_x$ and $d_y$ respectively. Then, $R(y)=\Res_x(P,Q)$ can be
  computed in $\softO(d_x^2d_y)$ field operations, and the degree of $R$ is
  bounded by $2d_xd_y$.
\end{lemma}

\begin{lemma}\label{lem:trivres}
  Let $P(x,y,z)$ and $Q(x,y,z)$ be two polynomials whose degrees in each
  variable are bounded by $d$. Then, $R(y,z)=\Res_x(P,Q)$ can be computed
  in $\softO(d^5)$ field operations, and the degree of $R$ in each
  variable is bounded by $2d^2$.
\end{lemma}
\begin{proof}
  The Sylvester matrix has at most $2d$ columns and its entries are
  bivariate polynomials whose degrees in $y$ and $z$ are bounded by $d$.
  Thus, its determinant is a polynomial whose degrees in $y$ and $z$ are
  bounded by $2d^2$.

  We first perform a Kronecker substitution by considering
  $\tilde{P}(x,y)=P(x,y,y^{2d^2+1})$ and
  $\tilde{Q}(x,y)=Q(x,y,y^{2d^2+1})$, which are polynomials of degrees
  $\le d$ in $x$ and $\le 2d^3+d$ in $y$. Note that the choice to replace $z$ by
  $y^{2d^2+1}$ is made to be able to invert the Kronecker substitution after 
  the resultant computation.

  Next, we compute $\tilde{R}(y) = \Res_x(\tilde{P}(x,y),
  \tilde{Q}(x,y))$.  By Lemma~\ref{lem:bivres}, it is a univariate
  polynomial of degree at most $4d^4+2d^2$ and can be computed in
  $\softO(d^5)$ operations.
  We can then invert the Kronecker substitution to get $R(y,z)$, which
  can be done in time linear in the number of monomials, that is in
  $O(d^4)$.
\end{proof}

\subsection{Non-generic situations}

Our analysis assumes in the first place that the $\ell$-torsion elements
are generic in a rather strong sense, see e.g. \cite[Def. 11]{expOg} for
details. This is expected to be the case
with overwhelming probability, when the base field is large enough and
the curve is taken at random in a large family. However, to obtain a
proven complexity we must also consider the cases where there exist
$\ell$-torsion elements that are non-generic. We follow the strategy
of~\cite{expOg} where another polynomial system is designed and solved
for each non-generic situation, for instance the fact that an 
$\ell$-torsion divisor is of weight less than 3, or that some points
involved in the modelling are not distinct while they generically are.
We do not give all the details, but the number of polynomial systems to
consider is bounded by a constant, and each of these polynomial systems
describes a
situation that is smaller than the generic one in the sense that it has either less variables
or a lower degree, so that the complexity bound is maintained.

\section{Bounds on the coefficients of $\psi$}\label{sec:boundabc}

The system of equations~\eqref{eq:relabcs} giving $\sigma_1$, $\sigma_2$
and $\sigma_3$ in terms of $a$, $b$, $c$ is homogeneous if we put weight
$1/2$ to $a$,
$b$, $c$ and $\sigma_1$, weight $1$ to $q$ and $\sigma_2$, weight
$3/2$ to $\sigma_3$, and weight $0$ to $\mu_0$, $\mu_1$, and $\mu_2$ so any
polynomial in a reduced Gröbner basis of the
corresponding ideal will have the same property.
Computing such a Gröbner basis with the lexicographical ordering
$a>b>c>\sigma_1>\sigma_2>\sigma_3>\mu_0>\mu_1>\mu_2>q$
(we did this computation with the Magma V2.23-4
software), we get a
polynomial $\Psi_c$ of degree $6$ in $c$ that does not involve $a$ or $b$, and
which has the following form:
\begin{small}
  $$\begin{array}{l}
  \Psi_c(q,c, \sigma_1,\sigma_2,\sigma_3, \mu_0, \mu_1, \mu_2) = 
  D(\mu_0, \mu_1, \mu_2)^3\, c^6 + \sum_{i=0}^5 \psi_c^{(i)}(q,\sigma_1,\sigma_2,\sigma_3, \mu_0, \mu_1,
  \mu_2)\, c^i,\end{array}
$$
\end{small}\leavevmode
where $D(\mu_0, \mu_1, \mu_2) = -27\,\mu_0^2 + 18\,\mu_0 \mu_1 \mu_2 - 4\, \mu_0 \mu_2^3 -
4\,\mu_1^3 + \mu_1^2\mu_2^2$ is the discriminant
of the polynomial $T^3+\mu_2\, T^2 + \mu_1\, T + \mu_0$.

By computing Gr\"obner bases for other lexicographical orderings (with
$a>c>b>\sigma_1>\sigma_2>\sigma_3>\mu_0>\mu_1>\mu_2>q$
and
$b>c>a>\sigma_1>\sigma_2>\sigma_3>\mu_0>\mu_1>\mu_2>q$
respectively), we obtain that
polynomials of the following form also belong to the ideal generated by the
polynomials in the system of equations~\eqref{eq:relabcs}:
\begin{small}
  $$\begin{array}{l}
  \Psi_b(q,b, \sigma_1,\sigma_2,\sigma_3, \mu_0, \mu_1, \mu_2) = 
  D(\mu_0, \mu_1, \mu_2)^3\, b^6 + \sum_{i=0}^5 \psi_b^{(i)}(q,\sigma_1,\sigma_2,\sigma_3, \mu_0, \mu_1,
  \mu_2)\, b^i,\\
  \Psi_a(q,a, \sigma_1,\sigma_2,\sigma_3, \mu_0, \mu_1, \mu_2) = 
  D(\mu_0, \mu_1, \mu_2)^3\, a^6 + \sum_{i=0}^5 \psi_a^{(i)}(q,\sigma_1,\sigma_2,\sigma_3, \mu_0, \mu_1,
  \mu_2)\, a^i.
\end{array}
$$
\end{small}\leavevmode

The polynomials $\psi_a^{(i)}$, $\psi_b^{(i)}$ and $\psi_c^{(i)}$ are
homogeneous of weighted degree $3-i/2$ with
respect to the grading given above.

\begin{lemma}\label{lem:boundabc}
    The absolute values of the coefficients $a, b, c$ of $\psi=a+b\eta+c\eta^2$
    are bounded above by $O(q^{1/2})$.
\end{lemma}

\begin{proof}
    First, we consider the equation $\Psi_c=0$. 
    We write $c=\widetilde c\,q^{1/2}$, $\sigma_1 = \widetilde{\sigma_1}\, q^{1/2}$,
    $\sigma_2 = \widetilde{\sigma_2}\, q$, $\sigma_3 = \widetilde{\sigma_3}\,
    q^{3/2}$.
  Since $\psi_c^{(i)}$ is homogeneous and has
  weighted degree $3-i/2$, there is a polynomial
  $\theta^{(i)}_c(\widetilde{\sigma_1}, \widetilde{\sigma_2},
  \widetilde{\sigma_3}, \mu_0,
  \mu_1,\mu_2)$ such that
  \begin{equation}\label{eq:hom}\psi_c^{(i)}\left(q,\sigma_1,\sigma_2, \sigma_3, \mu_0, \mu_1,
  \mu_2\right)\cdot c^i = q^3 \widetilde c^i\,\theta^{(i)}_c(\widetilde{\sigma_1}, \widetilde{\sigma_2},
  \widetilde{\sigma_3}, \mu_0,
  \mu_1,\mu_2).\end{equation}
    Weil's bounds imply that $\lvert\widetilde{\sigma_i}\rvert = O(1)$ for
    $i\in\{1,2,3\}$. Therefore, for all $i\in\{0,\ldots, 5\}$, we obtain that $\lvert \theta^{(i)}_c(\widetilde{\sigma_1}, \widetilde{\sigma_2},
  \widetilde{\sigma_3}, \mu_0,
  \mu_1,\mu_2)\rvert = O(1)$.
  For fixed $\mu_0, \mu_1,\mu_2\in\mathbb Q$ such that $\mu_0+\mu_1 T + \mu_2
  T^2 + T^3$ is the minimal polynomial of a totally real algebraic number, the
  discriminant $D(\mu_0, \mu_1, \mu_2)$ must be nonzero.
  Equations~$\Psi_c = 0$ and \eqref{eq:hom} imply the following inequality:
    $$\vert\widetilde c\rvert^6 -
    \sum_{i=0}^{5} \frac{\lvert \theta^{(i)}_c(\widetilde{\sigma_1}, \widetilde{\sigma_2},
  \widetilde{\sigma_3}, \mu_0, \mu_1,
  \mu_2)\rvert}{\lvert D(\mu_0, \mu_1, \mu_2)\rvert^3} \lvert
    \widetilde c\rvert^i\leq 0.$$
    Then $\vert\widetilde c\rvert$ must be smaller or equal to the
    largest root of this polynomial inequality, which can itself be
    bounded, for instance, with Cauchy's bound
    $$\lvert\widetilde c\rvert \leq 1+\max_{0\leq i\leq
    5}\left\{\frac{\lvert \theta^{(i)}_c(\widetilde{\sigma_1}, \widetilde{\sigma_2},
  \widetilde{\sigma_3},\mu_0, \mu_1,
  \mu_2)\rvert}{\lvert D(\mu_0, \mu_1, \mu_2)\rvert^3}\right\},$$
    which shows that $\lvert\widetilde c\rvert=O(1)$, and hence
    $\lvert c\rvert = O(q^{1/2})$.
    The proof for the bounds on $\lvert a\rvert$ and $\lvert b\rvert$ are
    similar, using the equations $\Psi_a = 0$ and $\Psi_b = 0$.
\end{proof}

\section{Small elements in ideals of $\Z[\eta]$}
\label{sec:smallgen}

We first recall that we consider only primes $\ell$ that do not divide the
discriminant of the minimal polynomial of $\eta$ (Condition (C2)). Hence,
if $\Z[\eta]$ is not the maximal order of $\mathbb Q(\eta)$, this has no
consequence on the factorization properties of $\ell$. 

\begin{lemma}\label{lem:smallgen}
    For any prime $\ell$ that splits completely in $\Z[\eta]$, each prime
    ideal $\gothp_i$ above $\ell$ contains a non-zero element $\alpha_i$ of the
    form $\alpha_i = a_i + b_i\eta +c_i\eta^2$, where $|a_i|$, $|b_i|$ and
    $|c_i|$ are integers in $O(\ell^{1/3})$, and the norm of $\alpha_i$
    is in $O(\ell)$.
\end{lemma}

\begin{proof}
  The coefficients of the elements of the ideal $\gothp_i$ represented by
  polynomials in $\eta$ form a lattice. Applying Minkowski's bound to
  this lattice, we obtain the existence of a non-zero element $\alpha_i =
  a_i + b_i\eta +c_i\eta^2$ in $\gothp_i$ for which the $L_2$-norm of
  $(a_i,b_i,c_i)$ is in $O(\ell^{1/3})$. From this bound on the
  $L_2$-norm, we derive a bound on the $L_\infty$-norm, and finally on
  the norm of $\alpha_i$ as an algebraic number. At each step, the
  constant hidden in the $O()$ gets worse but still depends only
  on $\mathbb Z[\eta]$.
\end{proof}

For any given $\eta$, it is not difficult to make the constants in the
$O()$ fully explicit. We do it in the particular case of
$\Z[\eta_7]$, with $\eta_7 = 2\cos(2\pi/7)$, which is the RM
used in our practical experiments.  Since $\Z[\eta_7]$ is a principal ring, a
more direct approach leads to bounds for a generator that are tighter
than what would be obtained by a naive application of the previous lemma.

\begin{lemma}\label{prop:main}
  Every ideal $\gothp_i$ of norm $\ell$ in $\Z[\eta_7]$ has a generator
  $\alpha_i$ of the form $a_i + b_i\eta_7 +c_i\eta_7^2$, where
  $a_i,b_i,c_i\in\Z$ satisfy
  $$
    \lvert a_i \rvert < 2.415\cdot\ell^{1/3}\, ;\quad
    \lvert b_i \rvert < 1.850\cdot\ell^{1/3}\, ;\quad
    \lvert c_i \rvert < 1.764\cdot\ell^{1/3}\, .
  $$
\end{lemma}

\begin{proof}
  By abuse of notation, we identify $\mathbb Q(\eta_7)$ with the algebraic
  number field $\mathbb Q[X]/(X^3+X^2-2X-1)$ and we let $\sigma_1$, $\sigma_2$, $\sigma_3$ be the three real embeddings of
  $\Q(\eta_7)$ in $\R$. Let $\epsilon_1 = 1-\eta_7^2$ and
  $\epsilon_2=1+\eta_7$ be a pair of fundamental units, and let $\mu_i$ be a
  generator of $\gothp_i$. The logarithmic embedding $\varphi: x\mapsto
  (\log\lvert\sigma_1(x)\rvert, \log\lvert\sigma_2(x)\rvert,
  \log\lvert\sigma_3(x)\rvert)$ sends the set of generators of $\gothp_i$
  to the lattice generated by $\varphi(\epsilon_1)$ and
  $\varphi(\epsilon_2)$ translated by $\varphi(\mu_i)$. Solving a CVP for
  the projection of $\varphi(\mu_i)$ on the plane where the 3 coordinates
  sum-up to zero, we deduce a unit $\xi_i$ such that
  $\alpha_i=\xi_i\mu_i$ is a generator whose real embeddings are bounded
  by
  $$
  \lvert \sigma_1(\alpha_i) \rvert \leq 2.247\cdot \ell^{1/3},\quad
  \lvert \sigma_2(\alpha_i) \rvert \leq 1.803\cdot \ell^{1/3},\quad
  \lvert \sigma_3(\alpha_i) \rvert \leq 2.247\cdot \ell^{1/3}\,.
  $$
  Writing $\alpha_i=a_i+b_i\eta_7+c_i\eta_7^2$, the real embeddings can
  also be expressed as $(\sigma_1(\alpha_i), \sigma_2(\alpha_i),
  \sigma_3(\alpha_i))^T = V \cdot (a_i, b_i, c_i)^T,$ where $V$ is the
  Vandermonde matrix of $(\sigma_1(\eta_7), \sigma_2(\eta_7),
  \sigma_3(\eta_7))$. A numerical evaluation of its inverse 
  allows to translate the bounds on $\sigma_1(\alpha_i)$,
  $\sigma_2(\alpha_i)$, $\sigma_3(\alpha_i)$ into the claimed bounds on
  $a_i$, $b_i$, $c_i$.
\end{proof}

\section{Bounding the degrees of Cantor's division polynomials in genus
3}\label{sec:cantor}

The purpose of this section is to prove the following lemma on the Cantor's
division polynomials, which are explicit formulas for the endomorphism
corresponding to scalar multiplication~\cite{Ca94}.

\begin{lemma}\label{lemme10}
In genus 3, the degrees of Cantor's $\ell$-division polynomials are bounded by
$O(\ell^2)$. 
\end{lemma}

In~\cite{Ca94}, there are exact formulas for the degrees of the leading
and the constant coefficients $d_3$ and $d_0$. However, there is no
formula or bounds for the degrees of the other coefficients of the
$\ell$-division polynomials. Still, our proof strongly relies on
\cite{Ca94} and we do not try to make it standalone: we assume that the
reader is familiar with this article and all references to expressions,
propositions or definitions in this proof are taken from this paper. 

For a polynomial $P$ whose coefficients are themselves univariate
polynomials, we denote by $\maxdeg(P)$ the maximum of the degrees of its
coefficients.

We first prove a bound on the degrees of the coefficients of the
quantities $\alpha_r$ and $\gamma_r$ defined in~\cite{Ca94}, from which
the wanted bounds will follow. The key tools are the recurrence formulas
(8.31) and (8.33) that relate quantities at index $r$ to quantities at
index around $r/2$, in a similar fashion as for the division polynomials
of elliptic curves.
More precisely, the following lemma shows that when the index $r$ is
(roughly) doubled, $\maxdeg \alpha_r$ and $\maxdeg \gamma_r$
are roughly multiplied by $4$, which leads to the expected
quadratic growth.

\begin{lemma}\label{lemme11}
  Let $\ell\ge 12$, and assume that for all $i\le (\ell+9)/2$ the degrees
  $\maxdeg \alpha_i$ and $\maxdeg\gamma_i$ are bounded by $C$, then 
  $\maxdeg \alpha_\ell$ and $\maxdeg\gamma_\ell$ are bounded by
  $4C+36\ell+108$.
\end{lemma}

\begin{proof}
  We first deal with the bound on $\maxdeg\gamma_\ell$.
  Let us consider $r$ and $s$ around $\ell/2$ such that $\ell=r+s-5$:
  we take either $r=s-3=\ell/2+1$ if $\ell$ is even, or $r=s-4=(\ell+1)/2$
  otherwise. 

  From Equations~(8.30) and (8.31), the degree of
  $\gamma_{\ell}[h]\psi_{s-r}\psi_{r-2}\psi_{s-2}\psi_{r-1}\psi_{s-1}$
  is that of the determinant of the matrix $\mathcal{E}_{rs}[h]$ defined
  by:
  \[
    \mathcal{E}_{rs}[h] =
    \begin{pmatrix}
      \alpha_{r-3}\alpha_{s}[0] & \alpha_{r-3}\alpha_{s}[1] & 
      \psi_{r-3}\psi_{s} & \gamma_{r-3}\gamma_{s}[h] \\

      \alpha_{r-2}\alpha_{s-1}[0] & \alpha_{r-2}\alpha_{s-1}[1] & 
      \psi_{r-2}\psi_{s-1} & \gamma_{r-2}\gamma_{s-1}[h] \\

      \alpha_{r-1}\alpha_{s-2}[0] & \alpha_{r-1}\alpha_{s-2}[1] & 
      \psi_{r-1}\psi_{s-2} & \gamma_{r-1}\gamma_{s-2}[h] \\

      \alpha_{r}\alpha_{s-3}[0] & \alpha_{r}\alpha_{s-3}[1] & 
      \psi_{r}\psi_{s-3} & \gamma_{r}\gamma_{s-3}[h] \\
    \end{pmatrix} .
  \]

  Therefore we have an expression for the degrees of the coefficients of
  $\gamma_{\ell}$ in
  terms of objects at index around $r$ and $s$:
  $$ \deg \gamma_{\ell}[h] \le \deg\det \mathcal{E}_{rs}[h] -
  \deg(\psi_{r-2}\psi_{s-2}\psi_{r-1}\psi_{s-1}).$$
  In this last formula, the factor $\psi_{s-r}$ has been omitted, because $s-r$
  is either 3 or 4, and by (8.17) this has non-negative degree in any case. 
  Thus, we simply bounded it below by $0$ in the previous inequality.
  Before entering a more detailed analysis, we use Equation~(8.8) to
  rewrite the first column with expressions for which we have exact
  formulas for the degree:
  \[ \mathcal{E}_{rs}[h] =
    \begin{pmatrix}
      \psi_{r-4}\psi_{s-1} & \alpha_{r-3}\alpha_{s}[1] & 
      \psi_{r-3}\psi_{s} & \gamma_{r-3}\gamma_{s}[h] \\

      \psi_{r-3}\psi_{s-2} & \alpha_{r-2}\alpha_{s-1}[1] & 
      \psi_{r-2}\psi_{s-1} & \gamma_{r-2}\gamma_{s-1}[h] \\

      \psi_{r-2}\psi_{s-3} & \alpha_{r-1}\alpha_{s-2}[1] & 
      \psi_{r-1}\psi_{s-2} & \gamma_{r-1}\gamma_{s-2}[h] \\

      \psi_{r-1}\psi_{s-4} & \alpha_{r}\alpha_{s-3}[1] & 
      \psi_{r}\psi_{s-3} & \gamma_{r}\gamma_{s-3}[h] \\
  \end{pmatrix}.\]

  The determinant of $\mathcal{E}_{rs}[h]$ is the sum of products of 4
  $\psi$ factors and 4 $\alpha$ or $\gamma$ factors. The degrees of the
  former are explicitly known, while by hypothesis we have
  upper bounds on the latter, since all the indices are at most
  $(\ell+9)/2$. We can then deduce an upper bound on the
  degree of this determinant. All the $\psi_i$ have indices with $i$ in
  the range $[r-4, s]$ (remember that $r\le s$), and since their degrees
  increases with the indices, we can upper bound the degree of the products of
  the four $\psi$ factors by $4\deg \psi_s$. 
  Therefore we have
  $$ \deg \det \mathcal{E}_{rs}[h] \le 4(\deg \psi_s + C).$$
  In order to deduce an upper bound on $\maxdeg \gamma_{\ell}$, it remains to
  get a lower bound on the degree of the
  $\deg(\psi_{r-2}\psi_{s-2}\psi_{r-1}\psi_{s-1})$ term, and again by
  monotonicity of the degree in the index, we lower bound it by $4\deg
  \psi_{r-2}$. So finally, we get
  $$ \maxdeg \gamma_{\ell} \le 4C + (\deg\psi_s^4 - \deg\psi_{r-2}^4).$$
  Using (8.16) and (8.17), we deduce that for all $k$, we have
  $\deg(\psi_k^2)=3(k^2-9)$ and substituting this value and the expression
  of $r-2$ and $s$ in term of $\ell$, we obtain
  $$ \deg\psi_s^4 - \deg\psi_{r-2}^4 = 
  \left\{\begin{array}{ll}
    30\ell+90 & \text{if $\ell$ is even,} \\
    36\ell+108 & \text{if $\ell$ is odd,} \\
  \end{array}\right. $$
  and the result follows for $\maxdeg \gamma_\ell$.

  The proof for $\maxdeg \alpha_\ell$ follows the same line.
  Using the matrix $\mathcal{F}_{rs}[h]$ defined in (8.32) in a similar
  way as we used the matrix $\mathcal{E}_{rs}[h]$ and with the help of
  the formula (8.33), we end up with the following bounds
  $$ \maxdeg \alpha_{\ell} \le
  \left\{\begin{array}{ll}
    4C + 30\ell - 30 & \text{if $\ell$ is even,} \\
    4C + 36\ell - 36 & \text{if $\ell$ is odd,} \\
  \end{array}\right. $$
  which are stricter than our target.
  
  Finally, the bound $\ell\ge 12$ is necessary to ensure that the
  quantities $r$ and $s$ are at least 5, as required in~\cite{Ca94} to
  apply the formulas (8.31) and (8.33).
\end{proof}

We can now finish the proof of Lemma~\ref{lemme10}. We define two
sequences $(\ell_i)_{i\ge0}$ and $(C_i)_{i\ge0}$ as follows: let
$\ell_0=12$ and let $C_0$ be a bound on the degrees of the coefficients of
all the $\alpha_i$ and $\gamma_i$ for $i\le \ell_0$. Then for all
$i\ge1$, we define the sequences inductively by
$$ \left\{ \begin{array}{l}
  \ell_{i+1} = 2\ell_i-9 \\
  C_{i+1} = 4C_i + 36 \ell_{i+1} + 108.\\
\end{array}\right. $$
By Lemma~\ref{lemme11}, for all $i$ and all $\ell \le \ell_i$, the
degrees $\maxdeg\alpha_\ell$ and $\maxdeg\gamma_\ell$ are bounded by $C_i$.
The expression $\ell_i = (\ell_0-9)2^i+9 = 3\cdot 2^i + 9$
can be derived directly from
the definition and substituted in the recurrence formula of $C_{i+1}$ to
get
$ C_{i+1} = 4C_i + 216\cdot 2^i + 432$.
This recurrence can be solved by setting $\Gamma_i = C_i + 108\cdot 2^i +
144$, so that $\Gamma_{i+1} = 4\,\Gamma_i$, and we obtain
$C_i = (C_0 + 252)\, 4^i - 108\cdot 2^i - 144$.
Finally, for any $\ell$, we select the smallest $i$ such that $\ell\le
\ell_i$. This value of $i$ is $\lceil \log_2((\ell-9)/3) \rceil$.
The corresponding
bound for $\maxdeg \alpha_\ell$ and $\maxdeg\gamma_\ell$ is then $C_i$,
which grows like $O(\ell^2)$ (and we remark that the effect of the
ceiling can make the constant hidden in the $O()$ expression grow by a factor at most 3).
\smallskip
 
Using the expression (8.10), we have $\maxdeg\delta_\ell\le
\maxdeg\alpha_\ell + \maxdeg\gamma_\ell$, and therefore the bound
$O(\ell^2)$ also applies to the degrees of the coefficients of
$\delta_\ell$. And using the formula (8.13), the same holds as well for
the coefficients of $\epsilon_\ell/y$.

This concludes the proof of Lemma~\ref{lemme10}.

\section{Experimental results}\label{sec:experiments}

In order to evaluate the practicality of our algorithm, we have tested it
on one of the families of genus-3 hyperelliptic curves having explicit RM
given in \cite[Theorem~1]{TTV}. Formulas for their RM endomorphisms are
described in~\cite{KoSm}: for $t\ne \pm 2$, the curve $\C_t$
with equation
$$ y^2 = x^7 -7x^5+14x^3-7x + t,$$
admits an endomorphism given in Mumford representation by
$$\eta_7(x,y)=\langle X^2+{11}\,xX/2+x^2-{16}/{9},y\rangle.$$
The fact that this expression has degree 2 while one would generically
expect a degree 3 is no accident: it comes from the construction
in~\cite{TTV} of the endomorphism as a sum of two automorphisms on a
double cover of the curve.
We have $\eta_7^3 + \eta_7^2-2\eta_7-1 = 0$, so that 
the ring $\Z[\eta_7]$ is isomorphic to the ring of integers $\Z[2\cos(2\pi/7)]$ of the real
subfield of the cyclotomic field $\Q(e^{2i\pi/7})$.
All the numerical data in this section have been obtained for the
parameter $t=42$, on the prime field $\F_p$ with $p=2^{64}-59$.

In our practical computations, the main differences with the theoretical
description are the following: we use Gröbner basis algorithms instead of
resultants, we consider also small non-split primes $\ell$ and small powers,
and we finish the computation with a parallel collision search. The
source code for our experiments is available at
\url{https://members.loria.fr/SAbelard/RMg3.tgz}.

\subsection{Computing modular information with Gröbner
basis}\label{sec:modinfoGB}

Although the polynomial system resolution using resultants has a
complexity in $\softO(\ell^4)$, the real cost for small values of $\ell$
is already pretty large. In the resolution method described in
Section~\ref{sec:resultants},
each bivariate resultant is computed by evaluation / interpolation and hence
requires the computation of many univariate resultants. We illustrate this by counting the number of
univariate resultants to perform and their degrees for the main step of
the resolution (the part that reaches the peak complexity). We also
measure the cost of such resultant computations using the NTL 10.5.0 and
FLINT 2.5.2 libraries, both linked against GMP 6, when the base field is
$\F_{2^{64}-59}$. These costs do not include the evaluation /
interpolation steps which might also be problematic for large instances,
because they are hard to parallelize.

\begin{center}
\begin{tabular}{|l|l|l|r|r|}
  \hline
  $\ell$ & \#res & Deg & Cost (NTL) & Cost (FLINT) \\
  \hline
  13  & 525M & 16,000 & 1,850 days    &  735 days  \\ 
  \hline
  29  & 12.8G & 80,000 & 310,000 days & 190,000 days \\
  \hline
\end{tabular}
\end{center}

We were more successful with the direct approach using Gröbner bases that
we now describe.
For computing the kernel of a given endomorphism, we
computed a Gröbner basis of the system~\eqref{eq:kernel_endo} with some small modifications. First, we
observe that the only occurrences of $y_1$ and $y_2$ are within the
monomial $y_1 y_2$. Consequently, we can remove one variable by replacing
each occurrence of $y_1 y_2$ by a fresh variable $y$.  Next, we need to
make the system $0$-dimensional by encoding the fact that $d_3(x_3)$ and
$\widetilde{d_3}(x_1,x_2)$ are nonzero. This is done by introducing
another
fresh variable $t$ and by adding the polynomial $S(x_1, x_2, x_3)t - 1$
to the system, where $S(x_1, x_2, x_3)$ is the squarefree part of
$d_3(x_3)\widetilde{d_3}(x_1, x_2)$.  Finally, since each
polynomial is symmetric with respect to the transposition of the
variables $x_1$ and $x_2$, we can rewrite the equations
using the symmetric polynomials $s_1 = x_1+x_2$ and $s_2=x_1\,x_2$. This
divides by two the degree in $x_1$ and $x_2$ of the equations.
We end-up with a system in 5 variables.

The whole construction can be slightly modified to compute the pre-image
of a given divisor by the endomorphism: to model $\alpha(D) = Q$, we
write $D=P_1+P_2+P_3-3\infty$ and solve for $\alpha(P_1-\infty) +
\alpha(P_2-\infty) = Q - \alpha(P_3-\infty)$. In that case, the variable
$y_3$ gets involved in all the equations, so that we get a system in 6
variables.

For $\ell=2$, the $2$-torsion elements are easily deduced from the
factorization of $f$, and by computing a pre-image of a $2$-torsion
divisor, we get a point in $J[4]$ from which we could deduce $a,b,c\bmod
4$. Dividing again by $2$ was too costly, due to the fact that the
$4$-torsion point was in an extension of degree 4. For
$\ell=3$, which is an inert prime, we ran the kernel computation for the
multiplication-by-3 endomorphism, without using the RM property.
The norm being 27, this is the largest modular computation that we performed
(and the most costly in terms of time and memory).
The prime $\ell=7$ ramifies in $\Z[\eta_7]$ as the cube of the ideal
generated by $\alpha_7 = -2-\eta_7+\eta_7^2$. The kernel of $\alpha_7$ can
be computed but it yields only one linear relation in $a,b,c\bmod 7$.
Dividing the kernel elements by $\alpha_7$ would give more information, but
again, this computation did not finish due to the field extension in
which the divisors are defined. The first split prime is $\ell=13$. We
use the following small generators: $(13) =
(2-\eta_7-2\eta_7^2)(-2+2\eta_7+\eta_7^2)(3+\eta_7-\eta_7^2)$, which
seem to produce the polynomial systems with the smallest degrees.
For instance, the apparently smaller element $1+\eta_7^2$ of norm 13
yields equations of much higher degrees $7, 71, 72, 73, 72$.
The next split prime is
29, which would maybe have been feasible, but was not necessary for our
setting.  In the following table,
we summarize the data for these systems, that were obtained with
Magma V2.23-4 on a Xeon E7-4850v3 at 2.20GHz, with 1.5 TB RAM.%
\footnote{The 
F4 algorithm can be highly sensitive to the modelling of the problem and we
refer to the source code.
In particular, thanks to serendipity, we saved a
factor greater than 12 in the runtime for $\ell=7,13$ by forgetting to take the squarefree part of the
saturation
polynomial. We have no explanation for this phenomenon.}

\begin{center}
  \begin{tabular}{|l|c|c|c|c|l|}
    \hline
    mod $\ell^k$ & \#var & degree of each eq. & time & memory & $a,b,c\bmod \ell^k$ \\
    \hline
    2 & --- & --- & --- & --- & $0,0,0$ \\
    4 (inert${}^2$)& 6 & $7,7,14,15,15,10$ & 1 min & negl. & $2,2,2$ \\
    3 (inert) & 5 & $7,53,54,55,26$ & 14 days & 140 GB & $1,2,1$ \\
    $7=\gothp_1^3$ & 5 & $7,35,36,37,36$ & 3.5h & 6.6 GB & {\small $a+2b+4c\equiv 2$}\\
    $13=\gothp_1\gothp_2\gothp_3$ & 5 & $7,44,45,46,52$ & $3\times3$ days & 41 GB &
    $12,10,9$\\
    $29=\gothp_1\gothp_2\gothp_3$ & 5 & $7,92,93,94,100$ &
    {\small $>\!3\!\times\!2$ weeks} & {\small $>\!0.8$ TB} & 
    --- \\
    \hline
  \end{tabular}
\end{center}

\subsection{Parallel collision search for RM curves}

The classical square-root-com\-plexity search in genus 3 requires $O(q)$ group
operations~\cite{Elkies98}. For RM curves, this can be improved by
searching for the coefficients $a, b, c$ of $\psi =
\pi+\pi^\vee$ in $\Z[\eta]$. This readily yields a complexity in
$O(q^{3/4})$, using the equation $a D + b\eta(D) + c\eta^2(D) = (q+1) D$,
that must be satisfied for any rational divisor $D$. While a baby-step
giant-step approach is immediate to design, it needs $O(q^{3/4})$ space
and this is the bottleneck. 
A low-memory, parallel version of this search can be obtained with the
algorithm of~\cite{LMPMCT}, where the details are given only for a
2-dimensional problem, while here this is a 3-dimensional problem. But we
did not hit any surprise when adapting the parameters to
our case.  Also, just like in~\cite{LMPMCT}, including some anterior
modular knowledge is straightforward: if $a,b,c$ are known modulo $m$,
the expected time is in $O(q^{3/4}/m^{3/2})$.

We wrote a dedicated C implementation with a few lines of assembly to
speed-up the additions and multiplications in $\F_p$, taking advantage of
the special form of $p$. This implementation performs
$10.7M$ operations in the Jacobian per second using 32
(hyperthreaded) threads of a 16-core bi-Xeon E5-2650 at 2 GHz.
We used the knowledge of $\psi$ modulo $156$ but not of the known
relation modulo 7 for simplicity (there is no obstruction to using it and
saving an additional $7^{1/2}$ factor).

After computing about 190,000 chains of
average length 32,000,000, we got a collision, from which we deduced
 $$ \psi = 2551309006 + 2431319810\ \eta_7 -847267802\ \eta_7^2,$$
and the coefficients of the characteristic polynomial $\chi_\pi$ of the Frobenius
are then
\begin{small}
  $$
  \sigma_1 = 986268198,\ \ \sigma_2 = 35389772484832465583,\ \
  \sigma_3 = 10956052862104236818770212244.$$
\end{small}\leavevmode
The number of group operations that were done is slightly less than
$43\, (p^{3/4}/156^{3/2})$. This factor 43 is close to the average that we
observed in our numerous experiments with smaller sizes. Scaled on a
single (physical) core, we can estimate the cost of this collision search 
to be 105 core-days.

\bibliographystyle{plain}
\bibliography{biblio}

\end{document}